\documentclass[12pt,reqno,a4paper]{amsart}
\usepackage{amsmath}

\usepackage{amsthm}
\usepackage{amssymb}
\usepackage{ifthen}
\usepackage{graphicx}
\usepackage[font=small,labelfont=bf]{caption}
\usepackage[subrefformat=parens]{subcaption}
\usepackage{epstopdf}
\nonstopmode \numberwithin{equation}{section}
\newtheorem{theorem}{Theorem}[section]

\newtheorem{remark}{Remark}[section]

\newtheorem{problem}{Problem}
\newtheorem{lemma}{Lemma}[section]
\newtheorem{corollary}{Corollary}[section]

\begin{document}
\date{}
\title{Vietoris type theorem related to positivity of trigonometric polynomials}
\author{
Priyanka Sangal
}
\address{Department of  Mathematics  \\
Indian Institute of Technology, Roorkee-247 667,
Uttarkhand,  India
}
\email{sangal.priyanka@gmail.com}

\author{
A. Swaminathan
}
\address{
Department of  Mathematics  \\
Indian Institute of Technology, Roorkee-247 667,
Uttarkhand,  India
}
\email{a.swaminathan@ma.iitr.ac.in, mathswami@gmail.com}

\begin{abstract}
In this work, a Vietoris type theorem for the positivity of sine and cosine sum for a particular
sequence of real numbers is provided. In this connection, the positivity of a particular type
of sine sum involving ratio of some parameters is given, which is new in the literature.
Various new results that follow from the Vietoris type theorem include improved estimates for the
location of the zeros of a class of trigonometric polynomials and new positive sums
for orthogonal polynomials. An open problem is also provided for
the partial sums of the generalized polylogarithm.
\end{abstract}

\maketitle

2010 Mathematics Subject Classification: {Primary 42A05; 42A32; Secondary 33C45; 26D05}

\keywords{Keywords: Positive trigonometric sums; Abel summation formula;
Location of zeros;
Komatu integral operator;
Chebyshev polynomials;
Orthogonal polynomials on the unit circle.}

\pagestyle{myheadings}\markboth{Priyanka Sangal and A.Swaminathan}
{Vietoris type theorem related to positivity of trigonometric polynomials}

\section{Preliminaries}
The partial sum of the Fourier series given by
\begin{align}\label{eqn:Fourier-partial-sum}
T_n(x)= \dfrac{a_0}{2} +\sum_{k=1}^{n} (a_k \cos kx+b_k \sin kx)
\end{align}
has several applications in Fourier analysis. In particular, for $b_k=0$ and $a_k=0$
in \eqref{eqn:Fourier-partial-sum}, respectively, we get the trigonometric cosine sum and the trigonometric
sine sum. Besides Fourier analysis, the positivity of these cosine sum and sine sum has interesting
applications and vast literature. Earlier results corresponding to related sine series appeared in 1755 and
for the cosine series appeared in 1848. For information in this direction, we refer to
\cite[Chapter 4, p.300]{milovanovi-mitrinovi-rassias-1994-book-polynomials} and references therein.
In this  manuscript we are interested in finding the positivity of cosine sum and sine sum for
particular case of the coefficients $a_n$ and $b_n$
that has $a_k=b_k=r_k$ where
\begin{align}\label{eqn:r_k-coefficients}
r_0=2,\quad  r_1=1 \quad r_k:=\frac{(k+\alpha)^{\lambda}}{(k+\beta)^{\mu}}, \qquad
k\geq 2, \quad \mbox{ where} \quad \alpha,\beta\geq 0.
\end{align}
Note that this way of considering same values for $a_k$ and $b_k$, $k\geq 2$ was first given by
Vietoris \cite{vietoris-1958} in which
$a_k=b_k=\gamma_k$, where $\gamma_k$ is defined as:
\begin{align*}
\gamma_0=\gamma_1=1\quad  \hbox{and} \quad \gamma_{2k}=\gamma_{2k+1}
    =\frac{1\cdot3\cdot5\cdots(2k-1)}{2\cdot4\cdot6\cdots2k}
=\dfrac{(1/2)_k}{k!}, \quad k=1,2,\ldots.
\end{align*}
Before Vietoris, the positivity of cosine and sine sum were considered separately.
Vietoris also gave sufficient conditions on the coefficients of a general class of sine and cosine sums
that ensure their simultaneous positivity in $(0,\pi)$ which is given by the following.
\begin{theorem}\rm{\cite{vietoris-1958}}\label{thm:vietoris-1958}
Suppose that $a_0\geq a_1 \geq a_2 \cdots \geq a_n >0$ and $2k a_{2k}\leq (2k-1)a_{2k-1}$, $k\geq1$,
then for all positive integers $n$ and $\theta\in(0,\pi)$, we have
\begin{align}\label{eqn:vietoris-sequence}
\sum_{k=1}^n a_k\sin{k\theta}>0 \qquad \hbox{  and  }  \qquad \sum_{k=0}^n a_k\cos{k\theta}>0.
\end{align}
\end{theorem}

The importance of Vietoris' inequalities is known after the work of Askey and Steinig
\cite{askey-steinig-1974-trig-sum-ams} where a simpler proof is given.
Furthermore, these results have applications including improved estimates for the
location of zeros of a class of trigonometric polynomials, new positive sums of orthogonal polynomials,
positive quadrature methods, derivation of various summation formula including hypergeometric summation formula.
For the applications of these inequalities, see \cite{askey-steinig-1974-trig-sum-ams, askey-1996-viet-ineq,
brown-hewitt-1984-pos-trig-sum, gasper-1969-nonneg-sum-JMAA}.
For other applications of positive trigonometric sums we refer to
\cite{dimitrov-merlo-2002-cons-approx, signal-processing-book, Fernandez-2004}
and references therein. It is important to note that as a recent improvement, the positive
trigonometric sine inequality was improved in \cite{kwong-2015-JAT}.
For the algebraic point of view of bounds of nonnegative trigonometric polynomials,
we refer to \cite{tkachev-arxiv-2003}. In 1995 Belov \cite{belov-1995-sine-sum} obtained
the necessary and sufficient condition for the positivity of trigonometric sine sum
in the interval $(0,\pi)$ which is also the sufficient condition for the cosine sum.

\begin{lemma}\rm{\cite{belov-1995-sine-sum}}\label{lemma:belov-1995-sine-sum}
Let $\{a_k\}_{k=0}^{\infty}$ be any decreasing sequence of positive real numbers. Then the condition
\begin{align*}
\sum_{k=1}^n (-1)^{k-1}ka_k\geq 0 ,\quad  \hbox{$\forall n\geq 2$, $a_1>0$},
\end{align*}
is necessary and sufficient for the validity of the inequality
\begin{align*}
\sum_{k=1}^na_k\sin{k\theta}>0, \quad \hbox{$\forall n\in\mathbb{N}$, $0<\theta<\pi$},
\end{align*}
and the same condition implies that,
\begin{align*}
\sum_{k=0}^n a_k\cos{k\theta}>0, \quad \mbox{$\forall n\in\mathbb{N}$, $0<\theta< \pi$}.
\end{align*}
\end{lemma}
Belov's result is the best possible for the positivity on $(0,\pi)$
of sine sums with nonnegative and decreasing sequence of coefficients.
This result is stronger than Theorem \ref{thm:vietoris-1958}.
For a complete account of details of the literature in this direction,
we refer to \cite{alzer-kwong, koumandos-2007-ext-viet-ramanujan,
saiful-swami-2011-CAMWA} and the references therein.
We will use Abel's summation formula as a tool in proving
many of our results including the main result of Section \ref{sec:positivity-2},
which is Theorem \ref{thm:new-positive-sine-sums}.
The statement of Abel summation formula is as follows:

\begin{lemma}[Abel's Summation formula]\label{lemma:abel-sum}
If $\{b_k\}_{k=0}^{\infty}$ and $\{c_k\}_{k=0}^{\infty}$ be two sequences of real numbers, then
\begin{align*}
\sum_{k=0}^n b_kc_k=\sum_{k=0}^{n-1}\left( \Delta b_k \sum_{j=0}^k c_j\right) + b_n\sum_{k=0}^n c_k,
\end{align*}
where $\Delta b_k=b_k-b_{k+1}$.
\end{lemma}

The rest of the manuscript is organized as follows.
In Section \ref{sec:positivity-3}, a generalization of Vietoris' inequalities is obtained.
In Section \ref{sec:positivity-2}, for a new sequence $\{r_k\}$ the positivity of trigonometric
sine and cosine sums are obtained which extend inequalities given in
\cite{acharya-thesis-1997} and \cite{saiful-swami-2011-CAMWA}.
The new sequence is related to generalized polylogarithm.
The application of new inequalities obtained in \ref{sec:positivity-3}
in finding the location of zeros of trigonometric polynomials
is discussed in Section \ref{sec:positivity-4}.
Further new inequalities involving orthogonal polynomials
on the unit circle are also provided in Section \ref{sec:positivity-4}.

\section{Extension of Vietoris' Inequalities}\label{sec:positivity-3}
In this section, our main objective is to find certain extension of Theorem
\ref{thm:vietoris-1958} by means of generalization of the coefficients.
This generalization leads finding the location of
zeros of certain trigonometric polynomials and new positive sums of certain orthogonal
polynomials to show the importance of these generalizations.

Among several extensions of Theorem \ref{thm:vietoris-1958}
available in the literature, the following
extension given in \cite{brown-dai-wang-2007-ext-viet-ramanujan}
is important for our discussion.
\begin{theorem}
\rm{\cite{brown-dai-wang-2007-ext-viet-ramanujan}}
\label{thm:koumandos-2007-ext-viet-ramanujan}
Let $b_k$ be defined as
\begin{align}\label{eqn:define-koum-coeff}
b_{2k}=b_{2k+1}=\frac{(1-\alpha)_k}{k!},\quad k\geq0
\end{align}
for $\alpha\in(0,1)$. Then, for any positive integer $n$ and $0<\theta<\pi$, the following
positivity condition holds.
\begin{align}\label{eqn:koum-cosine-sum}
\sum_{k=0}^n b_k\cos{k\theta}>0\quad \hbox{ for } \alpha_0 \leq\alpha<1,
\end{align}
where $\alpha_0\in (0.308443,0.308444)$
is the unique root in $(0,1)$ of the equation
\begin{align}\label{eqn:izumi_number}
\int_0^{3\pi/2}\frac{\cos{t}}{t^{\alpha}}dt=0.
\end{align}
Numerically $\alpha_0=0.3084437\ldots$. The sums in \eqref{eqn:koum-cosine-sum}
are unbounded below when $\alpha<\alpha_0$. Further, for the sine sums for
$\alpha_0 \leq \alpha<1$ the following hold.
\begin{enumerate}
\item $\displaystyle\sum_{k=1}^{2n} b_k \sin{k\theta}>0$ for $\theta\in\left(0,\pi-\dfrac{\pi}{n}\right)$,
\item $\displaystyle\sum_{k=1}^{2n+1} b_k \sin{k\theta}>0$ for $\theta\in(0,\pi)$.
\end{enumerate}
\end{theorem}

\begin{remark}
Using summation by parts, we can conclude that Theorem \ref{thm:koumandos-2007-ext-viet-ramanujan}
also holds good for the sequence $\{a_k\}$ satisfying
\begin{align*}
\frac{a_{2k}}{a_{2k-1}} \leq \left(1-\frac{\alpha}{k}\right).
\end{align*}
and $\alpha_0\leq \alpha\leq 1$.
Note that Theorem \ref{thm:koumandos-2007-ext-viet-ramanujan}
does not hold for any $\alpha$ with $\alpha<\alpha_0$.
\end{remark}

This remark leads to the following.
\begin{problem}\label{prob:vietoris-extension}
To find the sequence $\{a_k\}$ so that the positivity of cosine sums also holds for $\alpha$ with
$\alpha<\alpha_0$.
\end{problem}

To address Problem \ref{prob:vietoris-extension} we define
a new sequence $\{a_k\}$ satisfying for $b\geq c>0$,
\begin{align*}
\frac{a_{2k+1}}{a_{2k}} &\leq 1,\quad k=0,1,2\ldots n,\\
\frac{a_{2k}}{a_{2k-1}} &\leq \left(1-\frac{b-c}{b+n-k}\right)\left(1-\frac{\alpha}{k}\right),\quad k=1,2\ldots n.
\end{align*}
Thus, in accordance with \eqref{eqn:vietoris-sequence}, we define the sequence $\{c_k\}$ by
$a_k=c_k$ and $c_0=1$ such that
\begin{align}\label{dfn:ext-viet-coeff}
c_{2k}=c_{2k+1}=\frac{B_{n-k}}{B_n}\frac{(1-\alpha)_k}{k!}, \quad k=0,1,2,\ldots,n,
\end{align}
where
\begin{align}\label{eqn:define-Bk}
B_0=1\quad  \hbox{and}\quad  B_k=\frac{(b)_k}{(c)_k}\frac{1+b-c}{b},\qquad b\geq c>0 \qquad k\geq 1.
\end{align}
Now we are ready to provide our result for the positivity of cosine sums involving the coefficients $\{c_k\}$.

\begin{theorem}\label{thm:new-cosine-sum}
Let $\{c_k\}$ be defined as in \eqref{dfn:ext-viet-coeff}. Then for all positive integers $n$ and $0<\theta<\pi$,
we have
 \begin{align*}
 \sum_{k=0}^n c_k \cos{k\theta}>0
 \end{align*}
 when $\alpha\geq \alpha_0'$, where $\alpha_0'$ is the unique root in $(0,1)$ of the equation
 \begin{align}\label{eqn:integral-cosine-sum}
\int_0^{3\pi/2} t^{-\alpha} \cos{t}\left(1-\frac{2t}{3\pi}\right)^{b-c} dt=0.
\end{align}
This $\alpha_0'$ is related to the
Littlewood-Salem-Izumi constant $\alpha_0$ defined in \eqref{eqn:izumi_number}.
\end{theorem}
\begin{proof}
For $k=0,1,\ldots,n$, choose
$a_{2k}=a_{2k+1}=\dfrac{B_{n-k}}{B_n}$ and $b_k$ be defined as in
Theorem \ref{thm:koumandos-2007-ext-viet-ramanujan}.
Then,
\begin{align*}
\sum_{k=0}^n c_k\cos{k\theta}
=\sum_{k=0}^n a_k b_k \cos{k\theta}
=\sum_{k=0}^{n-1}(a_k-a_{k+1})\sum_{j=0}^k b_j \cos{j\theta} +a_n\sum_{j=0}^n b_j \cos{j\theta}.
\end{align*}
Using Theorem \ref{thm:koumandos-2007-ext-viet-ramanujan},
we have for $n\in\mathbb{N}$,
$\displaystyle\sum_{k=0}^n b_j \cos{j\theta}>0$, $\alpha\geq\alpha_0$.
Also it is clear from the definition of $a_k$, $a_{2k}-a_{2k+1}=0$.
Now it remains to show that $a_{2k-1}>a_{2k}$ for $k=1,2,\ldots,[n/2]$.
Clearly
\begin{align*}
a_{2k-1}-a_{2k} =\dfrac{(c)_n}{(b)_n} \left[\dfrac{(b)_{n-2k+1}}{(c)_{n-2k+1}}-\dfrac{(b)_{n-2k}}{(c)_{n-2k}} \right]
                =\dfrac{(c)_n}{(b)_n}\dfrac{(b)_{n-2k}}{(c)_{n-2k}}.\dfrac{b-c}{c+n-2k} >0.
\end{align*}

We proved the result only for $\alpha\geq \alpha_0$. It remains to prove the result
for $\alpha_0'\leq \alpha\leq \alpha_0$. Note that for $b=c=1$, $\alpha_0'=\alpha_0$ and hence
we consider the case $b\neq 1$ and $c \neq 1$.
This can be obtained by finding a relation between $\alpha_0$ and $\alpha_0'$ below.
Now an easy computation leads to
\newline
$\displaystyle
\lim_{n\rightarrow \infty} \left( \dfrac{\theta}{n}\right)^{1-\alpha}
 \sum_{k=0}^n \dfrac{B_{n-k}}{B_n} \dfrac{(1-\alpha)_k}{k!} \cos{\left(\dfrac{k\theta}{n}\right)}
$
\begin{align*}
=\dfrac{1}{\Gamma(1-\alpha)}\lim_{n\rightarrow \infty}\left(\dfrac{k\theta}{n} \right)^{-\alpha}
\left(
 \dfrac{b}{(1+b-c)}\dfrac{(c)_n}{(b)_n}
 +
 \sum_{k=0}^{n-1} \dfrac{\theta}{n}\dfrac{(b)_{n-k}}{(c)_{n-k}}\dfrac{(c)_n}{(b)_n}
\right)
    \cos{\left(\dfrac{k\theta}{n}\right)}
.
 \end{align*}
Using the fact that $\displaystyle\lim_{n\rightarrow \infty}\dfrac{(c)_n}{(b)_n}=0$ for $b>c$,
we get
\begin{align*}
\lim_{n\rightarrow \infty}\left(\dfrac{k\theta}{n} \right)^{-\alpha}
\left[
 \dfrac{b}{(1+b-c)}\dfrac{(c)_n}{(b)_n}
\right]
    \cos{\left(\dfrac{k\theta}{n}\right)}
\rightarrow 0.
 \end{align*}
This gives
\newline
$
\displaystyle
\lim_{n\rightarrow \infty} \left( \dfrac{\theta}{n}\right)^{1-\alpha}
 \sum_{k=0}^n \dfrac{B_{n-k}}{B_n} \dfrac{(1-\alpha)_k}{k!} \cos{\left(\dfrac{k\theta}{n}\right)}
$
\begin{align*}
&=\dfrac{1}{\Gamma(1-\alpha)}  \lim_{n\rightarrow \infty}\left(\dfrac{k\theta}{n} \right)^{-\alpha}
 \sum_{k=0}^{n-1} \dfrac{\theta}{n}\dfrac{(b)_{n-k}}{(c)_{n-k}}\dfrac{(c)_n}{(b)_n}
    \cos{\left(\dfrac{k\theta}{n}\right)}\\
&=\dfrac{1}{\Gamma(1-\alpha)}  \lim_{n\rightarrow \infty}  \sum_{k=1}^{n}\dfrac{\theta}{n} \dfrac{(b)_{k}}{(c)_{k}}
\dfrac{(c)_n}{(b)_n}\left((n-k)\dfrac{\theta}{n} \right)^{-\alpha} \cos{\left((n-k)\dfrac{\theta}{n}\right)}.
\end{align*}
Let $(n-k)\frac{\theta}{n}=t$ so that $-\frac{\theta}{n}dk=dt$.
This makes the right hand side of the above expansion as
\begin{align*}
\dfrac{1}{\Gamma(1-\alpha)} \lim_{n \rightarrow \infty} \int_{0}^{\left(1-\frac{1}{n}\right)\theta}
t^{-\alpha} \cos{t}\left(1-\dfrac{t}{\theta}\right)^{b-c} dt
=\dfrac{1}{\Gamma(1-\alpha)} \int_0^{\theta} t^{-\alpha}\cos{t}\left(1-\dfrac{t}{\theta}\right)^{b-c} dt.
\end{align*}
For $\theta=\dfrac{3\pi}{2}$, the resulting integral is \eqref{eqn:integral-cosine-sum}.
A computation using Mathematica software gives the solution of the integral
\eqref{eqn:integral-cosine-sum} as the zero of the hypergeometric function
\begin{align}\label{hypergeometric2F3:solution-integral}
{}_2F_3\left[\dfrac{1-\alpha}{2},1-\dfrac{\alpha}{2};\dfrac{1}{2},\dfrac{1}{2}(2-\alpha+b-c),
    \dfrac{1}{2}(3-\alpha+b-c);\dfrac{-9\pi^2}{16}\right]=0
\end{align}
Let the zero of \eqref{hypergeometric2F3:solution-integral} be $\alpha_0'$.
Note that, $b=c$ gives $\alpha_0$ as in Theorem \ref{thm:koumandos-2007-ext-viet-ramanujan}.
Here $\alpha_0$ is obtained by the integral
$\displaystyle \int_0^{3\pi/2}\frac{\cos{t}}{t^{\alpha}} dt $.

The solution of $\displaystyle \int_0^{3\pi/2} t^{-\alpha}\cos{t}
\left(1-\dfrac{2t}{3\pi}\right)^{b-c} dt$ is given by
\begin{align}\label{eqn:solution-P-integral}
\begin{split}
\mathcal{P}(\alpha,b-c)&:=\dfrac{\Gamma(1+b-c)\Gamma(1-\alpha)}{\Gamma(2-\alpha+b-c)}
\left(\frac{3\pi}{2}\right)^{1-\alpha}\\
&\qquad{}_2F_3\left[\frac{1-\alpha}{2},1-\frac{\alpha}{2};\frac{1}{2},
\frac{2-\alpha+b-c}{2},\frac{3-\alpha+b-c}{2};-\frac{9\pi^2}{16}\right]
\end{split}
\end{align}
and the solution of integral
$\int_0^{3\pi/2}\dfrac{\cos{t}}{t^{\alpha}} dt$ is given by
\begin{align}\label{eqn:solution-K-integral}
\mathcal{K}(\alpha)&:=\dfrac{\Gamma(1-\alpha)}{\Gamma(2-\alpha)}
\left(\frac{3\pi}{2}\right)^{1-\alpha}{}_2F_3\left[\frac{1-\alpha}{2},1-\frac{\alpha}{2};\frac{1}{2},
\frac{2-\alpha}{2},\frac{3-\alpha}{2};-\frac{9\pi^2}{16}\right]
\end{align}
Since $b=c$ gives $\mathcal{P}(\alpha,b-c)=\mathcal{K}(\alpha)$ we have the relation between
$\mathcal{P}(\alpha,b-c)$ and  $\mathcal{K}(\alpha)$  as
\begin{align}\label{eqn:relation-Priyanka-alpha-Koumandos}
\mathcal{P}(\alpha,b-c) = \mathcal{K}(\alpha) g(\alpha,b-c) +h(\alpha,b-c),
\end{align}
where $g(\alpha,b-c)$  and $h(\alpha,b-c)$ are to be determined.

In case of $g(\alpha,b-c)\equiv 1$ we have the resultant as
difference between \eqref{eqn:solution-P-integral} and
\eqref{eqn:solution-K-integral} which is given as
\begin{align*}
&\mathcal{P}(\alpha, b-c)-\mathcal{K}(\alpha)\\
&=\left(\frac{3\pi}{2}\right)^{1-\alpha}\Gamma(1-\alpha) \bigg[
\frac{\Gamma(1+b-c)}{\Gamma(2-\alpha+b-c)}
\sum_{k=0}^{\infty} \dfrac
{\left(\frac{1-\alpha}{2}\right)_k \left(\frac{2-\alpha}{2}\right)_k}
{\left(\frac{1}{2}\right)_k\left(\frac{2-\alpha+b-c}{2}\right)_k
\left(\frac{3-\alpha+b-c}{2}\right)_k k!}\left(-\frac{9\pi^2}{16}\right)^k\\
&\hspace{4cm}- \frac{1}{\Gamma(2-\alpha)}\sum_{k=0}^{\infty} \dfrac
{\left(\frac{1-\alpha}{2}\right)_k \left(\frac{2-\alpha}{2}\right)_k}
{\left(\frac{1}{2}\right)_k\left(\frac{2-\alpha}{2}\right)_k
\left(\frac{3-\alpha}{2}\right)_kk!}\left(-\frac{9\pi^2}{16}\right)^k\bigg]\\
&=\left(\frac{3\pi}{2}\right)^{1-\alpha}\Gamma(1-\alpha)
\sum_{k=0}^{\infty}\dfrac
{\left(\frac{1-\alpha}{2}\right)_k \left(\frac{2-\alpha}{2}\right)_k}
{\left(\frac{1}{2}\right)_k k!}\chi(\alpha,b-c)\left(-\frac{9\pi^2}{16}\right)^{k}
\end{align*}
where
\begin{align*}
\chi(\alpha,b-c)&=\dfrac{\Gamma(1+b-c)}{\Gamma(2-\alpha+b-c)}
\dfrac{1}{\left(\frac{2-\alpha+b-c}{2}\right)_k
\left(\frac{3-\alpha+b-c}{2}\right)_k}
-\dfrac{1}{\Gamma(2-\alpha)}\dfrac{1}{\left(\frac{2-\alpha}{2}\right)_k
\left(\frac{3-\alpha}{2}\right)_k} \\
&=\dfrac{\Gamma(1+b-c)}{\Gamma(2-\alpha+b-c)}\frac{2^{2k}\Gamma(2-\alpha+b-c)}
{\Gamma(2-\alpha+b-c+2k)}
-\dfrac{1}{\Gamma(2-\alpha)}\frac{2^{2k}\Gamma(2-\alpha)}
{\Gamma(2-\alpha+2k)}\\
&=2^{2k}\left[\dfrac{\Gamma(1+b-c)}{\Gamma(2-\alpha+b-c+2k)}
-\dfrac{1}{\Gamma(2-\alpha+2k)}\right]
\end{align*}
So $\mathcal{P}(\alpha, b-c)-\mathcal{K}(\alpha)$ becomes
\begin{align}
\label{eqn:h-alpha-bc}
\begin{split}
h(\alpha,b-c)&:=\left(\frac{3\pi}{2}\right)^{1-\alpha}\Gamma(1-\alpha)
\sum_{k=0}^{\infty}
\dfrac{\left(\frac{1-\alpha}{2}\right)_k \left(1-\frac{\alpha}{2}\right)_k2^{2k}}
{(1/2)_k k!}\\
&\left(\frac{\Gamma(1+b-c)}{\Gamma(2-\alpha+b-c+2k)}-\frac{1}
{\Gamma(2-\alpha+2k)}\right)
\left(-\frac{9\pi^2}{16}\right)^k
\end{split}
\end{align}

Thus the above computation gives the relation
\begin{align}\label{eqn:relation-two-integral}
\int_0^{3\pi/2} t^{-\alpha}\cos{t}
\left(1-\dfrac{2t}{3\pi}\right)^{b-c} dt=\int_0^{3\pi/2}\frac{\cos{t}}{t^{\alpha}} dt +
h(\alpha,b-c)
\end{align}
where $h(\alpha,b-c)$ is given in \eqref{eqn:h-alpha-bc}.
Note that from \eqref{eqn:relation-two-integral} we observe that,
at $b=c$, $h(\alpha,0)$ vanishes and the solution of \eqref{eqn:integral-cosine-sum} say $\alpha_0'$ is given by the
zero of the integral $\displaystyle\int_0^{3\pi/2}\frac{\cos{t}}{t^{\alpha}} dt$
which is $\alpha_0$.

Further at $b=c+1-\alpha_0$, simple computation using mathematica yields that
$h(0,1-\alpha_0)=1$ and $\int_0^{3\pi/2}\dfrac{\cos{t}}{t^{\alpha}} dt=-1$ at $\alpha=0$.
Hence, using \eqref{eqn:relation-two-integral}, $\alpha_0'=0$ is the solution of \eqref{eqn:integral-cosine-sum} if
$b=c+1-\alpha_0$. Moreover, $h(\alpha,b-c)$ does not have closed or simple form.

In case of $h(\alpha, b-c)\equiv 0$, we get that
\begin{align*}
g(\alpha, b-c) = \dfrac{\mathcal{P}(\alpha,b-c)}{\mathcal{K}(\alpha)}
\end{align*}
which for $b=c$ gives $g(\alpha, b-c) =1$. In this case also $g(\alpha, b-c)$
does not have a closed form.
\end{proof}

\begin{remark}
Using numerical values, the value of $\alpha_0'$ may be expressed in terms of
$\alpha_0$ and $b-c$ as
\begin{align*}
\alpha_0'=\alpha_0-\beta_0(b-c)-\beta_1(b-c)^2+O[b-c]^3.
\end{align*}
where the constants $\beta_0=0.4334739\ldots$ and $\beta_1=0.02203153\ldots$,
approximately.
\end{remark}

Writing $a_k=\dfrac{a_k}{c_k}c_k$ and using summation by parts, the following result is immediate.
\begin{corollary}\label{cor:new-cosine-sum}
Suppose that $a_0\geq a_1\geq \cdots \geq a_n>0$ and $(b+n-k)ka_k\leq (c+n-k)(k-\alpha)a_{k-1}$, $1\leq k\leq n$,
 then for all positive integers $n$,
  \begin{align*}
 \sum_{k=0}^n a_k \cos {k\theta}>0, \quad \hbox{ $0<\theta<\pi$},
 \end{align*}
 holds for $\alpha\geq \alpha_0'.$
\end{corollary}

Theorem \ref{thm:new-cosine-sum} motivates the following result which is the positivity of
corresponding sine sum, where the coefficients are given by \eqref{dfn:ext-viet-coeff}.

\begin{theorem}\label{thm:new-sine-sum}
Let $c_k$ be as in Theorem \ref{thm:new-cosine-sum}, then for $n\in {\mathbb{N}}$ and $0<\theta<\pi$, we have
\begin{align*}
\sum_{k=1}^{2n+1}c_k \sin{k\theta}>0, \qquad {\mbox{if, and only if,}} \quad \alpha\geq \alpha_0'
\end{align*}
and
\begin{align*}
\sum_{k=1}^{2n}c_k \sin{k\theta}>0 \qquad {\mbox{when}}
\qquad \alpha\geq \dfrac{3}{2}-\left(\dfrac{1+b}{2c}\right).
\end{align*}
\end{theorem}

\begin{proof}
Clearly
\begin{align*}
S_n(\theta)=\sum_{k=1}^n c_k \sin{k\theta}
\Longrightarrow S_{2n+1}(\pi-\theta)=2\sin{\dfrac{\theta}{2}}\sum_{k=0}^n
    e_k\cos\left(2k+\dfrac{1}{2}\right)\theta,
\end{align*}
where $e_k=c_{2k}=c_{2k+1}$, $k=0,1,2, \ldots, n$.
This gives $S_{2n+1}(\theta)>0$ for $0<\theta<\pi$ and $\alpha\geq\alpha_0'$.
On the other hand, for the even sine sum $S_{2n}(\theta)$,
in view of Lemma \ref{lemma:belov-1995-sine-sum},
$\displaystyle\sum_{k=1}^{2n}(-1)^{k-1}kc_k\geq 0$ implies
\begin{align*}
&\sum_{k=0}^{n-1}c_k-2n c_n=\sum_{k=0}^{n-1}
    \dfrac{B_{n-k}}{B_n}\dfrac{(1-\alpha)_k}{k!}-2n\dfrac{B_0}{B_n}\dfrac{(1-\alpha)_n}{n!}\\
&=\sum_{k=0}^{n-2}(B_{n-k}-B_{n-k-1})
    \sum_{j=0}^{k}\dfrac{(1-\alpha)_j}{j!}
        +B_1\sum_{k=0}^{n-1}\dfrac{(1-\alpha)_k}{k!}-2nB_0\dfrac{(1-\alpha)_n}{n!}\\
&\geq \dfrac{b}{c}\sum_{k=0}^{n-1}\dfrac{(1-\alpha)_k}{k!}-2n\dfrac{b}{1+b-c}\dfrac{(1-\alpha)_n}{n!}\\
&=\dfrac{b}{c}\dfrac{(2-\alpha)_{n-1}}{(n-1)!}-2\dfrac{b}{1+b-c}\dfrac{(1-\alpha)_n}{(n-1)!}\geq0,\quad
\hbox{only if $\alpha\geq \dfrac{3}{2}-\left(\dfrac{1+b}{2c}\right)$}.
\qedhere
\end{align*}
\end{proof}
The interpretation of Theorem \ref{thm:new-cosine-sum} and
Theorem \ref{thm:new-sine-sum} in
geometric function theory reflected as a concept of generalized stable functions by the
authors separately.
Theorem \ref{thm:new-sine-sum} leads to the following remark.
\begin{remark}
Note that
for the range $\alpha_0'\leq\alpha<\frac{3}{2}-\left(\frac{1+b}{2c}\right)$,
using the techniques similar to \cite{brown-dai-wang-2007-ext-viet-ramanujan} gives
\begin{align*}
S_{2n}(\theta)>0 \hbox{\quad for \quad  } 0<\theta\leq\pi-\dfrac{\pi}{2n}
\end{align*}
which is a better range of $\alpha$ than the one available in literature.
On the other hand, for $b=1+\beta$ and $c=1$
in Theorem \ref{thm:new-sine-sum}, we have $S_{2n}(\theta)>0$ for $2\alpha\geq 1-\beta$. This
result is different from the one given in \cite{saiful-2012-stable-results-in-math} due to the
difference in the coefficients $c_k$. Further, if $\beta=0$, then Theorem \ref{thm:new-sine-sum}
reduces to Vietoris result \cite{vietoris-1958} of $S_{2n}(\theta)>0$
for $\alpha \geq 1/2$ which is the best possible so far. However, if
$\beta\geq1$ then
Theorem \ref{thm:new-sine-sum} gives $S_{2n}(\theta)>0$ for all $\alpha\in(0,1)$ and $\theta\in(0,\pi)$
which improves the range for $\alpha$ than the one available in the literature.
\end{remark}

The following corollary of Theorem \ref{thm:new-cosine-sum} is immediate consequence of
Lemma \ref{lemma:abel-sum} and the non-negativity of
$\displaystyle\sum_{k=0}^n \dfrac{B_{n-k}}{B_n}\dfrac{(1-\alpha)_k}{k!}\cos{k\theta}$ for $\alpha\geq\alpha_0$.
\begin{corollary}\label{cor:sigma-bc-cosine}
Suppose that $e_0\geq e_1\geq \cdots \geq e_n>0$ and $(b+n-k)ke_k\leq (c+n-k)(k-\alpha)e_{k-1}$, $1\leq k\leq n$,
 then for all positive integers $n$ and $\alpha\geq \alpha_0$,
 \begin{align}\label{eqn:}
 \sum_{k=0}^n e_k \cos\left(2k+\dfrac{1}{2}\right)\theta>0, \quad \hbox{ $0<\theta<\pi$}.
 \end{align}
\end{corollary}
A different extension of Theorem \ref{thm:vietoris-1958} is addressed in Section \ref{sec:positivity-2}.

\section{Positive Trigonometric sums for Generalized Coefficients}\label{sec:positivity-2}
The objective of finding conditions on $\{r_k\}$ defined by \eqref{eqn:r_k-coefficients}
such that the corresponding sine sum is positive which is given as below.
\begin{theorem}\label{thm:b-by-c}
If $\alpha,\beta>0$ and $\lambda,\mu>0$ such that
$\alpha<\beta$, $\mu\geq 1+\lambda$ and $\lambda\beta-\alpha\mu<0$ then the following
holds.
\begin{align}\label{eqn:lambda-Nr-and-mu-Dr}
\tilde{S}_n(x)=\sin x+\sum_{k=2}^n \dfrac{(k+\alpha)^{\lambda}}{(k+\beta)^{\mu}}\sin kx>0, \quad x\in(0,\pi).
\end{align}
\end{theorem}
\begin{proof}
Let $r_k=\dfrac{(k+\alpha)^{\lambda}}{(k+\beta)^{\mu}}$ with $\alpha$, $\beta$ $\lambda$ and $\mu$ satisfying
the hypothesis of the theorem. Clearly \eqref{eqn:lambda-Nr-and-mu-Dr} is true for $n=1$.
For $n=2$,
\begin{align*}
\tilde{S}_2(x)=\sin x+\dfrac{(2+\alpha)^{\lambda}}{(2+\beta)^{\mu}}\sin 2x
>\sin x\left(1-2\dfrac{(2+\alpha)^{\lambda}}{(2+\beta)^{\mu}}\right)>0\quad \hbox{ for $\mu\geq \lambda+1$ }.
\end{align*}
For $k\geq 3$ we observe that
$r_k$ is decreasing, which is obtained by taking the logarithmic derivative of $r_k$ to get
\begin{align*}
r_k'=r_k\left[\frac{(\lambda-\mu)k+(\lambda\beta-\mu\alpha)}{(k+\alpha)(k+\beta)}\right]
\end{align*}
which is clearly negative.
Writing $P_k:=\dfrac{(1+1/(k+\alpha))^{\lambda}}{(1+1/(k+\beta))^{\mu}}$
and taking logarithmic derivative we get that
\begin{align*}
\frac{P_k'}{P_k}=\frac{\mu}{(k+\beta)(1+k+\beta)}-\frac{\lambda}{(k+\alpha)(1+k+\beta)}.
\end{align*}
This expression is positive only if
$\mu(k+\alpha)(k+1+\alpha)-\lambda(k+\beta)(k+1+\beta)>0$ or equivalently
\begin{align*}
\frac{\lambda}{\mu}<\frac{(k+\alpha)(k+1+\alpha)}{(k+\beta)(k+1+\beta)}<1
\end{align*}
which is true by hypothesis. This gives
$\displaystyle \dfrac{(k+\alpha)^{\lambda}}{(k+\beta)^{\mu}}-\dfrac{(k+1+\alpha)^{\lambda}}{(k+1+\beta)^{\mu}}$
is decreasing. Hence we have
\begin{align*}
\frac{(k+\alpha)^{\lambda}}{(k+\beta)^{\mu}}-2\frac{(k+1+\alpha)^{\lambda}}{(k+1+\beta)^{\mu}}
+\frac{(k+2+\alpha)^{\lambda}}{(k+2+\beta)^{\mu}}>0
\end{align*}
leading to the observation that $r_k$ is convex. Thus we can write
\begin{align*}
\tilde{S}_n(x)>\left[1-2\frac{(2+\alpha)^{\lambda}}{(2+\beta)^{\mu}}+\frac{(3+\alpha)^{\lambda}}{(3+\beta)^{\mu}}\right]\sin x
+\frac{(n+\alpha)^{\lambda}}{(n+\beta)^{\mu}}\frac{\sin nx}{2}.
\end{align*}
Using the fact that $\dfrac{(n+\beta)^{\mu}}{(n+\alpha)^{\lambda}}>n$ and
$1-2\dfrac{(2+\alpha)^{\lambda}}{(2+\beta)^{\mu}}+\dfrac{(3+\alpha)^{\lambda}}{(3+\beta)^{\mu}}>\dfrac{1}{4}$ we get
that $\tilde{S}_n(x)>0$ for $\pi/n\leq x\leq \pi-\pi/n$ and $n\geq 3$.

For $\pi-\pi/n<x<\pi$, substituting $x=\pi-t$ we get $0<t<\pi/n$. Then we have
\begin{align*}
\tilde{S}_n(x)=\tilde{S}_n(\pi-t)
=\sin{t}+t^{\lambda+\mu}\sum_{k=2}^n (-1)^{k-1}\frac{(kt+\alpha t)^{\lambda}}{(kt+\beta t)^{\mu}}\sin{kt}
\end{align*}
Writing $kt=\theta$ and $A=\dfrac{\left(1+\frac{\alpha}{k}\right)^{\lambda}}{\left(1+\frac{\beta}{k}\right)^{\mu}}$,
it is easy to see that the function $f(\theta)=A\dfrac{\sin{\theta}}{\theta^{\mu-\lambda}}$
is positive and
\begin{align*}
f'(\theta)=A\left(\frac{\cos{\theta}}{\theta^{\mu-\lambda}}
-\frac{(\mu-\lambda)\sin{\theta}}{\theta^{\mu-\lambda+1}}\right)<0.
\end{align*}
Therefore $f(\theta)$ is a decreasing function of $\theta\in(0,\pi)$.
Note that $f(\pi)>0$.
Also, for odd $k$, $(-1)^{k-1}\dfrac{(kt+\alpha t)^{\lambda}}{(kt+\beta t)^{\mu}}\sin kt>0$.
Using this with the fact that
\begin{align*}
\frac{((2k-1)t+\alpha t)^{\lambda}}{((2k-1)t+\beta t)^{\mu}}\sin(2k-1)t
-\frac{(2kt+\alpha t)^{\lambda}}{(2kt+\beta t)^{\mu}}\sin2kt>0
\end{align*}
we get $\tilde{S}_n(x)>0$ for $n$ odd. Similar argument leads to concluding the result
for $n$ is even. Combining all these cases we get $\tilde{S}_n(x)>0$ for all $n$ and the
proof is complete.
\end{proof}

If we choose $\lambda<0$ and $\mu>0$ then proceeding in a similar way
as in the proof of Theorem \ref{thm:b-by-c} the following result can be obtained.
\begin{theorem}\label{thm:new-positive-sine-sums}
For $\alpha\geq0$, $\beta\geq0$ and $\lambda,\mu>0$ such that $\lambda+\mu\geq1$,
\begin{align}\label{eqn:posi-sin-sum-qk}
\hat{S}_n(x)=\sin{x}+\sum_{k=2}^n \frac{\sin{kx}}{(k+\alpha)^{\lambda}(k+\beta)^{\mu}}>0,
\quad \mbox{for\, } x\in(0,\pi).
\end{align}
\end{theorem}

\begin{proof}
The positivity of the sine polynomial \eqref{eqn:posi-sin-sum-qk}
follows easily by comparing it with the well-known Fe\'jer-Jackson polynomial \cite[eqn. 24.1]{koumandos-survey}, namely
\begin{align*}
\sum_{k=1}^n \dfrac{\sin kx}{k}, \qquad x\in(0,\pi),
\end{align*}
as it can be easily shown that the ratios of the corresponding coefficients
\begin{align*}
\left(\dfrac{1}{(k+\alpha)^{\lambda}(k+\beta)^{\mu}}\right)/\left(\dfrac{1}{k}\right)
\end{align*}
form a decreasing sequence. Therefore, \eqref{eqn:posi-sin-sum-qk} can be proved easily using
the Abel summation formula and positivity of  Fe\'jer-Jackson
polynomial. We omit the details of the proof.
\end{proof}

For the positivity of the corresponding cosine polynomial $\hat{C}_n(x)$ we use the
following lemma given in \cite{saiful-swami-2011-CAMWA}.

\begin{lemma}
\label{lemma:saiful-positivity}
\cite{saiful-swami-2011-CAMWA}
Let $\alpha\geq0,\lambda\geq1, b_0=2, b_1=1$ and $b_k=\dfrac{1}{(k+\alpha)^{\lambda}}$,
$k\in\mathbb{N},k\geq2$, then for all $n\in\mathbb{N}$, the following inequalities
hold.
\begin{align*}
\frac{b_0}{2}+\sum_{k=1}^n b_k\cos{k\theta}>0 \quad \hbox{and} \quad \sum_{k=1}^nb_k\sin{k\theta}>0,
\quad \hbox{ for $0<\theta<\pi$.}
\end{align*}
\end{lemma}

However, due to its special nature compare to the results exist in the literature, we provide it as a theorem
with the outline of the proof. We use the following notation in the next theorem.
\begin{align} \label{eqn:q_k-coefficients}
q_0&=2,\quad  q_1=1 \quad q_k:=\frac{1}{(k+\alpha)^{\lambda}(k+\beta)^{\mu}}, \qquad
k\geq 2, \quad \mbox{ where} \quad \alpha,\beta\geq 0.
\end{align}
\begin{theorem}\label{thm:new-positive-cosine-sum}
Suppose that $\alpha\geq0$, $\beta\geq 0$ and $\lambda,\mu \geq 0$ such that $\lambda+\mu\geq1$
then,
\begin{align*}
\hat{C}_n(x):= \frac{q_0}{2}+\sum_{k=1}^n q_k \cos{k \theta} &>0,\quad \hbox{for $0< \theta<\pi$.}
\end{align*}
\end{theorem}
\begin{proof}
Let $n\in \mathbb{N}$ and $0<\theta<\pi$, using Lemma \ref{lemma:abel-sum}
$\hat{C}_n(x)$ can be rewritten as
\begin{align*}
&\hat{C}_n(x)=1+\cos{\theta}
+\sum_{k=2}^n \frac{1}{(k+\alpha)^{\lambda}(k+\beta)^{\mu}}\cos{k\theta}\\
&=\left[1-\frac{1}{(2+\alpha)^{\lambda}} \right](1+\cos{\theta})+\frac{1}{(n+\alpha)^{\lambda}}
\left[ 1+\cos{\theta}+\sum_{k=2}^n\frac{1}{(k+\beta)^{\mu}}\cos{k\theta}\right]\\
&\quad \quad +\sum_{k=2}^{n-1}\left[\left(\frac{1}{(k+\alpha)^{\lambda}}-\frac{1}{(k+1+\alpha)^{\lambda}}\right)
\left(1+\cos{\theta}+\sum_{j=2}^k\frac{\cos{j\theta}}{(j+\beta)^{\mu}} \right) \right]\\
&>0 \quad \hbox{for $\mu\geq 1$ and $\lambda\geq0$ using Lemma \ref{lemma:saiful-positivity}}.
\end{align*}
Proceeding in the similar fashion, we can also prove that it also holds for $\mu\geq0$ and $\lambda\geq1$.
Now for the case $0\leq \mu <1$ and $0\leq \lambda<1$ such that $\mu+\lambda\geq 1$,
the result follows from Belov's criterion \cite{belov-1995-sine-sum}.
\end{proof}

Note that $q_{n+1}=0$ for $\hat{C}_n(x)$. This means
\begin{align*}
\Delta^2q_{n-1}=q_{n+1}-2q_n+q_{n-1}=\frac{1}{(n-1+\alpha)^{\lambda}(n-1+\beta)^{\mu}}
-\frac{2}{(n+\alpha)^{\lambda}(n+\beta)^{\mu}}
\end{align*}
fails to be non-negative and hence
$\{q_k\}$ does not satisfy the conditions of
\rm{\cite[Theorem 1.2.8]{milovanovi-mitrinovi-rassias-1994-book-polynomials}}
for $n>2$, which implies that
Theorem \rm{\ref{thm:new-positive-cosine-sum}} cannot be obtained using
the Fej\'er criterion \rm{\cite[p.310]{milovanovi-mitrinovi-rassias-1994-book-polynomials}}.
On the other hand, Abel summation formula and Lemma \ref{lemma:saiful-positivity} guarantees
the positivity of $\hat{C}_n(x)$.
However, both procedures will provide the positivity of $\hat{S}_n(x)$.
Hence it is suggested that finding the reasons and an improved procedure may formulate
another direction of research in the positivity of trigonometric sums.

To provide graphical illustration to the interested readers, specific values
of the parameters are considered. Here we chose
$\alpha=.2$, $\beta=.4$, $\lambda=.3$, $\mu=.7$  and $\theta\in(0,\pi)$ for various values of $n$.
Figure \ref{fig:1}\subref{fig1a} and Figure \ref{fig:1}\subref{fig1b}
shows $\hat{C}_n(x)$ and $\hat{S}_n(x)$ for
$n=20$(continuous line), $n=30$(dash or hyphens) and $n=40$(dotted lines)
for the above values $\alpha=.2$, $\beta=.4$, $\lambda=.3$, $\mu=.7$.

\begin{figure}
\begin{subfigure}[b]{.4\linewidth}
\centering
\includegraphics[width=\linewidth]
{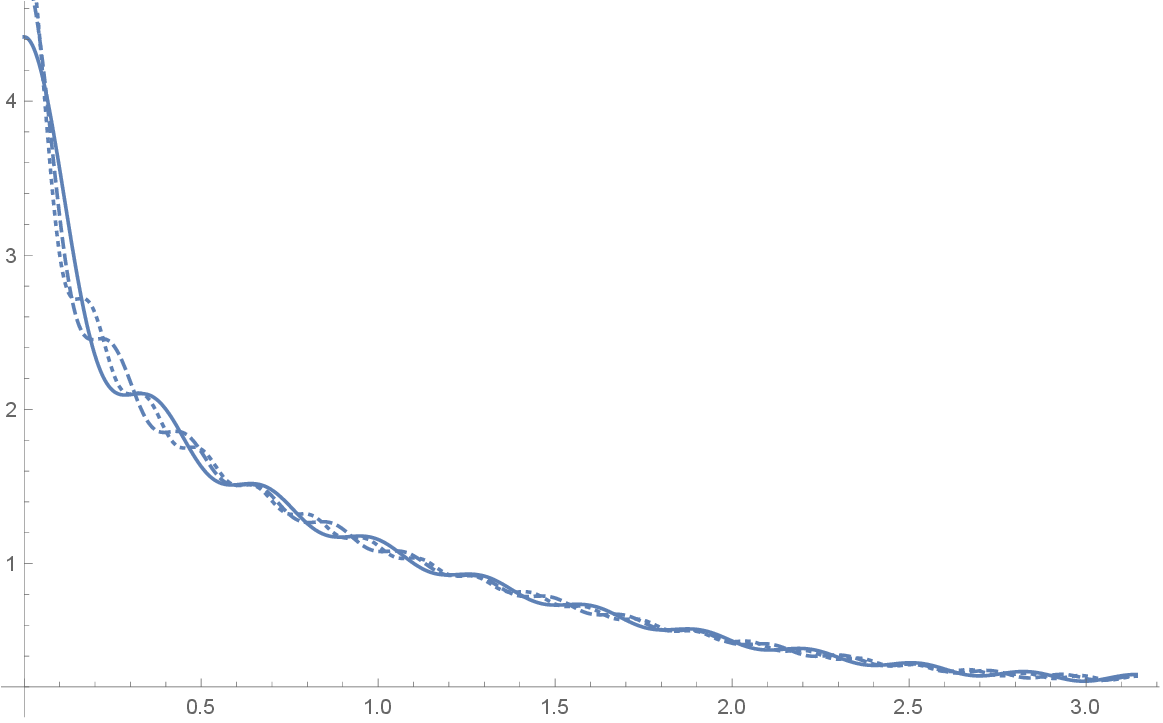}
\caption{}\label{fig1a}
\end{subfigure}\hfill
\begin{subfigure}[b]{.4\linewidth}
\centering
\includegraphics[width=\linewidth]{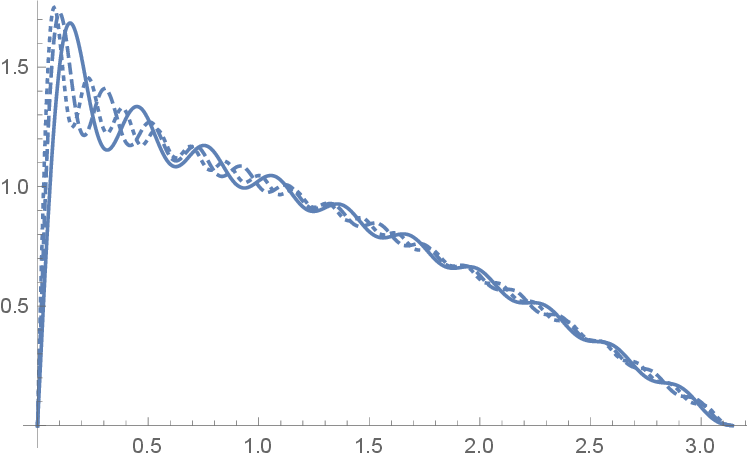}
\caption{}\label{fig1b}
\end{subfigure}%
\caption{\subref{fig1a}: $\hat{C}_n(x)$ and \subref{fig1b}:
$\hat{S}_n(x)$ for $n=20, 30, 40$.}
\label{fig:1}
\end{figure}

Similarly, Figure \ref{fig:2}\subref{fig2a} shows the graph of $\displaystyle\sum_{k=0}^{n} q_kz^k$
for $\alpha=.2$, $\beta=.4$, $\lambda=.3$, $\mu=.7$   and $|z|=1$, when
$n=20$(continuous line), $n=30$(dash or hyphens) and $n=40$(dotted lines). In
Figure \ref{fig:2}\subref{fig2b} the graph of $\displaystyle\sum_{k=0}^{n} q_kz^k$
for $\alpha=.2$, $\beta=.4$, $\lambda=.3$, $\mu=.7$   and $|z|=1$, when
$n=75$(continuous line), $n=100$(dash or hyphens) and $n=125$(dotted lines) is illustrated.

\begin{figure}
\begin{subfigure}[b]{.4\linewidth}
\centering
\includegraphics[width=\linewidth]{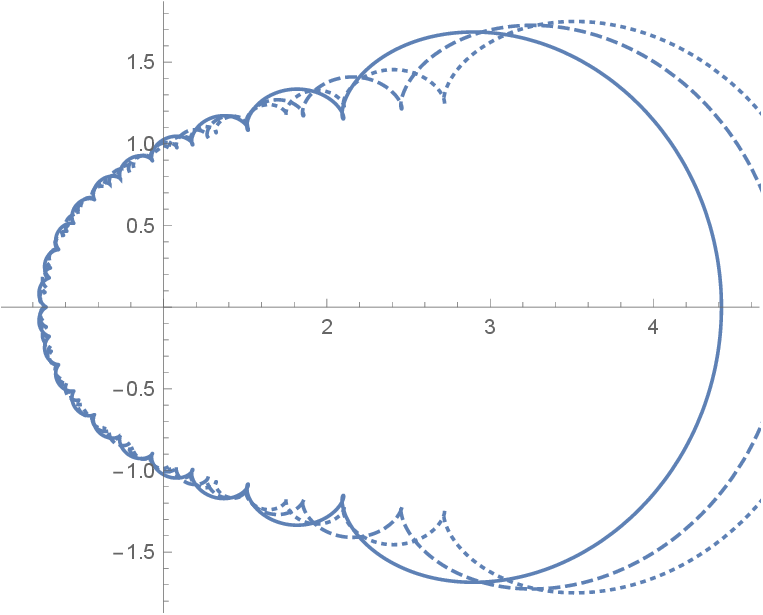}
\caption{}\label{fig2a}
\end{subfigure}\hfill
\begin{subfigure}[b]{.45\linewidth}
\centering
\includegraphics[width=\linewidth]{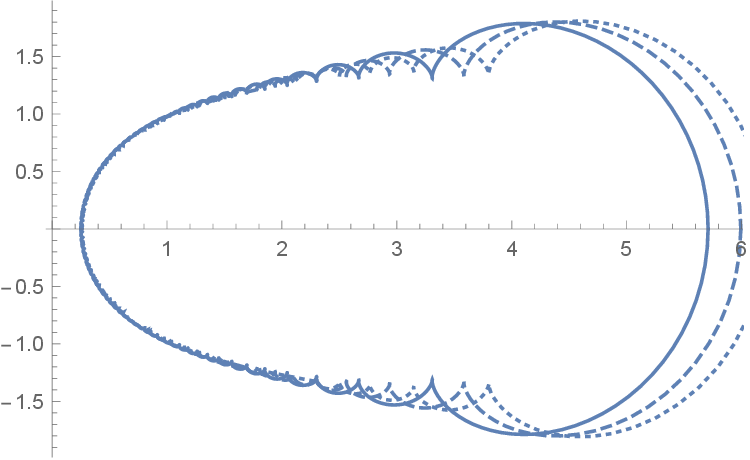}
\caption{}\label{fig2b}
\end{subfigure}%
\caption{$\displaystyle\sum_{k=0}^{n} q_kz^k$ for $\alpha=.2,\beta=.4,\lambda=.3,\mu=.7$  and $|z|=1$
with \subref{fig2a} for $n=20,30, 40$ and \subref{fig2b} for $n=75, 100, 125$.}
\label{fig:2}
\end{figure}

Note that a conclusion similar to Theorem \ref{thm:b-by-c} cannot be made for cosine sums with the same procedure.
Nevertheless, Theorem \ref{thm:b-by-c} is of much importance since there is no such result available in the
literature with the coefficients in the numerator having the parameter involving $k$ as in \eqref{eqn:r_k-coefficients}.
Further, similar to the positivity results involving the coefficients given in \eqref{eqn:q_k-coefficients},
positivity of sine sums and cosine sums involving the coefficients given in Theorem \ref{thm:b-by-c}
will also have wide applications.

A closely related result involving Jacobi polynomials is given in \cite{lewis-1979-siam}.
Generalization of the polynomials given in \cite{sangal-swaminathan-geom-sigma-bc} in line with
\cite{Muir} requires establishing the positivity results similar to the one given in \cite{lewis-1979-siam}
by extending the coefficients to the one given in Theorem \ref{thm:b-by-c}. More specifically,
it leads to finding conditions on $\delta$, $\lambda$, $\alpha$, $\beta$ such that
\begin{equation}\label{myeqn}
\sum_{k=0}^n
\frac{(1+\lambda)_{n-k}}{(1+\delta)_{n-k}}
\frac{(1+\lambda)_k}{(1+\delta)_k}
\frac{P_k^{(\alpha,\beta)}(x)}{P_k^{(\alpha,\beta)}(1)}
z^k \neq 0,
\end{equation}
where $P_k^{(\alpha,\beta)}(x)$ is the well known Jacobi polynomial.

If we substitute $\delta=1$ in ($\ref{myeqn}$), then we get
$$
\sum_{k=0}^n
\frac{(1+\lambda)_{n-k}}{(1)_{n-k}}
\frac{(1+\lambda)_k}{(1)_k}
\frac{P_k^{(\alpha,\beta)}(x)}{P_k^{(\alpha,\beta)}(1)}
z^k \neq 0,
$$
which, for the values, $0\leq \lambda \leq \alpha+\beta$, $\alpha\geq \beta > -\infty$,
is true. See \cite{lewis-1979-siam} for details.
Hence a
result similar to Theorem \ref{thm:b-by-c} for cosine sums will be of equal interest.
Imposing specific monotonic condition on the coefficients $\{a_k\}$ in
Theorem \ref{thm:new-positive-sine-sums} and Theorem \ref{thm:new-positive-cosine-sum},
summation by parts yields the following result.
\begin{corollary}\label{cor:new-positive-sum}
Let $\alpha\geq0,\beta\geq0$ and $\lambda\geq0,\mu\geq0$ such that $\lambda+\mu\geq 1$.
Also suppose that $\{a_k\}$ be a sequence of positive real numbers such that
\begin{align*}
(k+1+\alpha)^{\lambda}(k+1+\beta)^{\mu}a_{k+1} \leq (k+\alpha)^{\lambda}(k+\beta)^{\mu}
a_k\leq \cdots \leq (2+\alpha)^{\lambda}(2+\beta)^{\mu}a_2 \leq a_1
\end{align*}
and $a_1 \leq \dfrac{a_0}{2}$ holds.
Then, for $0<\theta<\pi$ and $n\in\mathbb{N}$, the following inequalities hold:
\begin{align*}
\frac{a_0}{2}+\sum_{k=1}^n a_k \cos{k\theta}>0 \quad and \quad \sum_{k=1}^n a_k \sin{k\theta}>0.
\end{align*}
\end{corollary}

This corollary has many interesting applications particularly in
finding the geometric properties like univalency, starlikeness, convexity
and close-to-convexity of analytic functions in the unit disc and this objective is addressed,
together with the details on these technical terms, in
\cite{sangal-swaminathan-geom-sigma-bc}.

We provide the next result as the generalization of \cite[Theorem 2.2]{saiful-swami-2011-CAMWA} on the
monotonicity of the cosine sums.
\begin{theorem}\label{thm:monotonicity-alpha-beta-cosine-sum}
Let $\alpha\geq0$, $\beta\geq 0$ and $\lambda\geq0,\mu\geq0$, then for every positive
integer $n$, we have
\begin{align}\label{eqn:new-mono-sum}
\frac{d}{d\theta}\left(\cos{\frac{\theta}{2}}\left(1+\cos{\theta}+\sum_{k=2}^n \frac{\cos{k\theta}}
{k(k+\alpha)^{\lambda}(k+\beta)^{\mu}} \right) \right)<0,\quad \mbox{for $0<\theta<\pi$}.
\end{align}
\end{theorem}

\begin{proof}
To prove \eqref{eqn:new-mono-sum}, it is sufficient to show the positivity of
\begin{align*}
\frac{1}{2}\sin{\frac{\theta}{2}}\left(1+\cos{\theta}+\sum_{k=2}^n
\frac{\cos{k\theta}}{k(k+\alpha)^{\lambda}(k+\beta)^{\mu}} \right)+\cos{\frac{\theta}{2}}
\left(\sin{\theta}+\sum_{k=2}^n \frac{\sin{k\theta}}{(k+\alpha)^{\lambda}(k+\beta)^{\mu}}\right).
\end{align*}

Note that the case $\lambda+\mu=0$ coinciding with the case $\gamma=1$ given in
\cite[Lemma 2.3]{saiful-swami-2011-CAMWA}.
Hence it remains to prove the theorem for $\lambda +\mu >0$. We consider this in two parts, viz.,
$0<\lambda+\mu<1$ and $\lambda+\mu \geq 1$.

Clearly,  from Theorem \ref{thm:new-positive-sine-sums}, we have
$\sin{\theta}+\displaystyle\sum_{k=2}^n \frac{\sin{k\theta}}{(k+\alpha)^{\lambda}(k+\beta)^{\mu}}>0$ for
$0<\theta<\pi$ and $\lambda\geq0$, $\mu\geq0$ such that $\lambda+\mu\geq1$. So it remains to prove the
positivity of $1+\cos{\theta}+\displaystyle\sum_{k=2}^n\frac{\cos{k\theta}}{k(k+\alpha)^{\lambda}(k+\beta)^{\mu}}$
for $\lambda+\mu\geq 1$ which, using Lemma \ref{lemma:abel-sum},  can be written as
\begin{align*}
&1+\cos{\theta} +\sum_{k=2}^n\frac{\cos{k\theta}}{k(k+\alpha)^{\lambda}(k+\beta)^{\mu}} \\
&=(1+\cos{\theta})\left( 1-\frac{1}{(2+\alpha)^{\lambda}(2+\beta)^{\mu}}\right)
    +\frac{1}{(n+\alpha)^{\lambda}(n+\beta)^{\mu}}
\left(1+\sum_{k=1}^n \frac{\cos{k\theta}}{k}\right)\\
&+\sum_{k=2}^{n-1}\left(\left( \frac{1}{(k+\alpha)^{\lambda}(k+\beta)^{\mu}}
-\frac{1}{(k+1+\alpha)^{\lambda}(k+1+\beta)^{\mu}}\right)\left(1+\cos{\theta}
+\sum_{j=2}^k\frac{\cos{j\theta}}{j} \right)\right)\\
&>0 \qquad \hbox{using Theorem \ref{thm:new-positive-cosine-sum}, \, for}
    \quad \lambda+\mu\geq 1, \quad 0<\theta<\pi.
\end{align*}

For proving the case $0<\lambda+\mu<1$,
let us suppose $b_0=b_1=1$ and $b_k=\dfrac{1}{(k+\alpha)^{\lambda}(k+\beta)^{\mu}}$
for $k\geq 2$. Then \eqref{eqn:new-mono-sum} can be rewritten as:
 \begin{align*}
& \frac{d}{d\theta}\left(\cos{\frac{\theta}{2}}\left(1+\cos{\theta}+\sum_{k=2}^n \frac{\cos{k\theta}}
{k(k+\alpha)^{\lambda}(k+\beta)^{\mu}} \right) \right)\\
&=\frac{d}{d\theta}\left(\sum_{k=1}^{n-1}(b_k-b_{k+1})\left[\left(1+\sum_{j=1}^k
\frac{\cos{j\theta}}{j}\right)\cos{\frac{\theta}{2}}\right]
+b_n\cos{\frac{\theta}{2}}\left(1+\sum_{j=1}^n\frac{\cos{j\theta}}{j}\right)\right)\\
&=\sum_{k=1}^{n-1}(b_k-b_{k+1})\frac{d}{d\theta}\left[\left(1+\sum_{j=1}^k
\frac{\cos{j\theta}}{j}\right)\cos{\frac{\theta}{2}}\right]
+b_n\frac{d}{d\theta}\left[\cos{\frac{\theta}{2}}\left(1+\sum_{j=1}^n
\frac{\cos{j\theta}}{j}\right)\right]\\
&<0 \quad \hbox{using \cite[Lemma 2.3]{saiful-swami-2011-CAMWA} for $0<\theta<\pi$}. \qedhere
\end{align*}
\end{proof}

For $\lambda=\mu=0$, \eqref{eqn:new-mono-sum} reduces to the
special case proved in \cite{saiful-swami-2011-CAMWA}.
Further, Theorem \ref{thm:monotonicity-alpha-beta-cosine-sum} can be regarded
as the generalization of the particular case of
Theorem \ref{thm:vietoris-1958}, with $a_k =1/k$.

\section{Applications to location of zeros, orthogonal polynomials and generalized polylogarithms}
\label{sec:positivity-4}
\subsection{Applications in estimating the location of zeros of certain trigonometric polynomials}

Suppose that the sequence $\{e_k\}$ satisfy the conditions of Corollary \ref{cor:sigma-bc-cosine},
that is
\begin{align}\label{eqn:1}
\frac{e_k}{e_{k-1}}\leq \left(\frac{c+n-k}{b+n-k}\right)\frac{(k-\alpha)}{k}, \qquad  1\leq k\leq n.
\end{align}
Then it follows from \eqref{eqn:} that
\begin{align}\label{eqn:2}
 \sum_{k=0}^n e_k \cos\left(k+\frac{1}{4}\right)\theta>0, \quad \hbox{ $0<\theta<2\pi$}.
\end{align}
With the transformation $\theta$ by $2 \pi-\theta$, this is equivalent to
\begin{align}\label{eqn:3}
\sum_{k=0}^n e_k \sin\left(k+\frac{1}{4}\right)\theta>0, \quad \hbox{ $0<\theta<2\pi$}.
\end{align}
Now using the identities,
\begin{align*}
\cos{(k+\lambda)\theta}=\cos\left(k+\frac{1}{4}\right)\theta \cos\left(\lambda-\frac{1}{4}\right)\theta
-\sin\left(k+\frac{1}{4}\right)\theta \sin\left(\lambda-\frac{1}{4}\right)\theta,\\
\sin{(k+\mu)\theta}=\sin\left(k+\frac{1}{4}\right)\theta \cos\left(\mu-\frac{1}{4}\right)\theta
+\cos\left(k+\frac{1}{4}\right)\theta \sin\left(\mu-\frac{1}{4}\right)\theta,
\end{align*}
we can easily get that, \eqref{eqn:2} and \eqref{eqn:3} imply,
\begin{align}
\sum_{k=0}^n e_k \cos(k+\lambda)\theta>0,
    \quad \hbox{for $0\leq \lambda \leq \frac{1}{4}$  and $0<\theta<2\pi$},\label{eqn:4}\\
\sum_{k=0}^n e_k \sin(k+\mu)\theta>0,
    \quad \hbox{for $\frac{1}{4}\leq \mu \leq \frac{1}{2}$  and $0<\theta<2\pi$}.\label{eqn:5}
\end{align}
Now by considering \eqref{eqn:2}-\eqref{eqn:5}, we can find the location of the zeros of the polynomials
\begin{align*}
p(\theta)=\displaystyle\sum_{k=0}^n a_k \cos{(n-k)\theta} \quad \hbox{and}\quad
q(\theta)=\displaystyle\sum_{k=0}^{n-1} a_k \sin{(n-k)\theta},
\end{align*}
where $a_0>a_1\geq a_2\geq\cdots\geq a_n>0$
by assuming the condition \eqref{eqn:1}. This condition is weaker than the conditions
exist in the literature in \cite{brown-hewitt-1984-pos-trig-sum},
which are otherwise weaker than the sharp estimates obtained in \cite{askey-steinig-1974-trig-sum-ams}.
Since the procedure is same as given in \cite{askey-steinig-1974-trig-sum-ams}, we omit the details.

\subsection{Positivity of Some Orthogonal Polynomial Sums}
We begin with the following observation.
\begin{remark}
If we substitute $x=\arccos t$, $t\in[-1,1]$ in Theorem \ref{thm:new-positive-cosine-sum},
then $\hat{C}_n(x)\geq0$ reduces to
\begin{align*}
T_0(t)+T_1(t)+\sum_{k=2}^n \frac{1}{(k+\alpha)^{\lambda}(k+\beta)^{\mu}}T_k(t)\geq0
\end{align*}
where $T_k(t)$ is the well known Chebyshev polynomial of degree $k$.
\end{remark}

This remark provides insight to search for applications of positivity of trigonometric sums to
the sequence of orthogonal polynomials. At first, we provide a result for Gegenbauer polynomials,
for the parameters considered in this paper, in similar lines to the results exist in the literature for
the sake of completeness. Motivated by the results on Gegenbauer polynomials and Chebyshev polynomials
we provide positivity results related to orthogonal polynomials on the unit circle which proposes a
new direction for future research.

\subsubsection{Gegenbauer Polynomial Sums}
Let $C_k^{\lambda}(x)$ be the Gegenbauer polynomial of degree $n$ and order $\lambda>0$ is
given by the generating function
\begin{align*}
(1-2xr+r^2)^{-\lambda}=\sum_{n=0}^{\infty}C_n^{\lambda}(x)r^n, \quad |x|<1.
\end{align*}
This power series and Gegenbauer polynomials occur so frequently and are of much interest.
By using the location of zeros of Gegenbauer polynomials, in \cite{lewis-1979-siam}
the zero-free property of Jacobi polynomial sums is obtained
which is further generalization of a result given in
\cite{ruscheweyh-1978-kakeya-thm-SIAM} on Gegenbauer polynomial sums.
Fejer \cite{fejer-1931-ultasperical-pol} proved that the power series coefficient
of $(1-r)^{-1}(1-2xr+r^2)^{-\lambda}$ are positive
in $0<\lambda\leq1/2$ i.e.

\begin{align*}
\sum_{k=0}^n C_k^{\lambda}(x)>0, \quad -1<x<1, 0<\lambda\leq 1/2.
\end{align*}
This result was extended in \cite[Theorem 4]{askey-steinig-1974-trig-sum-ams}.
For the normalized Gegenbauer polynomials $\dfrac{C_k^{\lambda}(x)}{C_k^{\lambda}(1)}$,
positivity result was established in \cite{brown-koumanods-wang-1998-ultraspherical-poly} as follows.
For $-1<x<1$ and for all positive integer $n$,
\begin{align}\label{eqn:koum:normalized-gegenbauer}
\sum_{k=0}^n \frac{C_k^{\lambda}(x)}{C_k^{\lambda}(1)}>0 \quad \hbox{for $\lambda\geq \lambda'=\alpha+1/2$},
\end{align}
and $\alpha'$ is the solution of $\int_0^{j_{\alpha,2}}t^{-\alpha}J_{\alpha}(t)dt=0$, where
$J_{\alpha}$ is the Bessel function of first kind of order $\alpha$ with $j_{\alpha,2}$ is the
second positive root. Numerically $\lambda'=0.23061297\ldots.$ When $\lambda<\lambda'$ the sums are
unbounded below in $(-1,1)$. One such extension is given below.

\begin{corollary}\label{cor:normalized-gegenbauer-positivity}
Let $\{a_k\}$ be as in Corollary \ref{cor:sigma-bc-cosine}, then for all positive $n$
and $-1<x<1$ we have
\begin{align*}
\sum_{k=0}^n a_k \frac{C_k^{\lambda}(x)}{C_k^{\lambda}(1)}>0\quad \hbox{ for all
$\lambda>0$.}
\end{align*}
\end{corollary}

\begin{proof}
Consider $\lambda\geq \lambda'$ then using summation by parts and
\eqref{eqn:koum:normalized-gegenbauer}
yields the required result. We omit the details of the proof.
Therefore, Corollary \ref{cor:normalized-gegenbauer-positivity} extends
Theorem 5 of \cite{askey-steinig-1974-trig-sum-ams}. \qedhere
\end{proof}

\subsubsection{Orthogonal Polynomials on the unit circle(OPUC)}
In \cite{kiran-ccs-SIGMA}, an illustration
of Orthogonal polynomials on the unit circle are provided in terms of the Gaussian hypergeometric
function ${}_2F_1(a,b;c;z)$. Taking lead from the results given in \cite{kiran-ccs-SIGMA} we provide
the following result.
\begin{theorem}
For $b>-1/2$, we have
\begin{align*}
\sum_{k=0}^n F_k^b(\omega)>0, |\omega|<1
\end{align*}
where $F_k^b(\omega)$ are the para orthogonal polynomial on the unit circle defined by the
generating function
\begin{align*}
{}_2F_1(b+1,-;-;\omega z)\cdot {}_2F_1(\bar{b}+1,-;-;z)
=\sum_{n=0}^{\infty} F_k^b(\omega)z^n , \quad z\in\mathbb{D},
\end{align*}
where $F_k^b(\omega):=\dfrac{(b+\bar{b}+1)}{k!}{}_2F_1(-k,b+1;b+\bar{b}+2;1-z)$.
\end{theorem}
\begin{proof}
To prove this result we use the concept of absolutely monotonicity of functions, where
a function is absolutely monotonic if its power series has non-negative coefficients.

Now, for real $b$, we consider the function $(1-z)^{-1}(1-(\omega+1)z+\omega z^2)^{-(b+1)}$
so that
\begin{align*}
\frac{1}{(1-z)}\frac{1}{(1-(\omega+1)z+\omega z^2)^{(b+1)}}
=\sum_{n=0}^{\infty}\left(\sum_{k=0}^n F_k^b(\omega)\right)z^n.
\end{align*}
Hence to prove the result, it is enough to prove that
the function $z\rightarrow (1-z)^{-1}(1-\omega z)^{-(b+1)}(1-z)^{-(\bar{b}+1)}$
is absolutely monotonic.

Let $\psi(z)=\dfrac{1}{(1-z)}\dfrac{1}{(1-(\omega+1)z+\omega z^2)^{b+1}}$, then
\begin{align*}
g(z)&=\log \psi(z)=-(b+2)\log(1-z)-(b+1)\log(1-\omega z)\\
g'(z)&=\dfrac{b+2}{1-z}+\frac{(b+1)\omega}{1-\omega z}=\sum_{n=0}^{\infty}(1+(b+1)\omega^{n+1})z^n
\end{align*}
i.e. $z\rightarrow g'(z)$ is absolutely monotonic function. Now $g(0)=0$, so
$g(z)$ is also absolutely monotonic function. Hence, $\psi(z)=e^{g(z)}$ is also
absolutely monotonic function leading to $\displaystyle
\sum_{k=0}^n F_{k}^b(\omega)>0 $ and the proof is complete.
\end{proof}
The following corollary can be obtained using summation by parts.
\begin{corollary}\label{cor:OPUC-positivity}
Let $\{a_k\}$ be as in Corollary \ref{cor:sigma-bc-cosine}, then for all positive $n$
and $|\omega|<1$ we have
\begin{align*}
\sum_{k=0}^n a_k F_k^b(\omega)>0\quad \hbox{ for all $b>-1/2$.}
\end{align*}
\end{corollary}

\subsection{An open problem on generalized polylogarithm}
The generalized polylogarithm \cite{saiful-swami-polylog} are be defined by the normalized function
\begin{align}\label{defn-polylog}
\Phi_{p,q}(a,b;z)= 1+ \sum_{k=2}^\infty \frac{(1+a)^p(1+b)^q}{(k+a)^p(k+b)^q} z^k
\end{align}
where $k+a\neq 0$, $k+b\not= 0$ and  $p, q \in \mathbb{C},$  $ {\rm Re\;}p>0$,
$ {\rm Re\;}q>0$.
These generalized polylogarithm are generalization of logarithm function and Riemann zeta functions
and can be expressed in terms of generalized hypergeometric functions \cite{saiful-swami-polylog}.

Note that the coefficients of generalized polylogarithm are constant multiples of $\{q_k\}$
given in \eqref{eqn:q_k-coefficients}. Hence we are interested in finding the positivity of
this generalized polylogarithm. This, on the other hand, gives the property that $\Phi_{p,q}(a,b;z)$
is in the class of Caratheod\`ary functions. The geometric properties for these generalized polylogarithm are
not completely studied. However closely related to the generalized polylogarithm is the Komatu
integral operator \cite{komatu-operator-1987-Kodai} which
is one of the well known integral operator in geometric function theory.
Positivity of trigonometric polynomials obtained in
Theorem \ref{thm:new-positive-sine-sums} and Theorem \ref{thm:new-positive-cosine-sum}
are useful in discussing the positivity of the generalized Komatu integral operator.
In other words, we use the generalized Komatu operator to study the positivity of partial sum of the
generalized polylogarithm given by \eqref{defn-polylog}.
In the integral representation of $\Gamma(p)$ given by
$ \displaystyle \Gamma(p)=\int_0^{\infty} e^{-t}t^{p-1} dt $ ,
changing $t=(n+a)u$ gives
\begin{align}\label{eqn-lerch-integral}
\dfrac{1}{(n+a)^p} = \dfrac{1}{\Gamma(p)}\int_0^{\infty} e^{-u(n+a)}u^{p-1} du.
\end{align}
Multiplying both sides of $\eqref{eqn-lerch-integral}$  by $z^n$, taking summation from $n=1,2, \ldots$ and then changing the variable $e^{-u}=s$, we get
\begin{align*}
z \Phi(z,r,\alpha):= \sum_{n=1}^{\infty} \dfrac{1}{\Gamma(p)} \int_0^1 \left(\log 1/s \right)^{p-1} s^a \dfrac{z}{1-zs}ds.
\end{align*}
Multiplying the resultant of replacing $a$ by $b$ and $p$ by $q$ in $\eqref{eqn-lerch-integral}$ with $\eqref{eqn-lerch-integral}$
 and using appropriate substitution gives
\begin{align*}
\Phi_{p,q}(a,b;z) = \dfrac{(1+a)^p(a+b)^q}{\Gamma(p)\Gamma(q)} \int_0^1 \int_0^1 \left(\log\, 1/s\right)^{p-1} \left(\log\, 1/t\right)^{q-1}
\dfrac{z s^a t^b}{1-stz} ds dt,
\end{align*}
where $a, \, b, \,  p-1, \, q-1 \in \{ z\in {\mathbb{C}}: {\rm Re \,} z>-1 \}$.
If we choose
\begin{align}\label{eqn: generalized Komatu operator}
\eta(t):=\left\{
           \begin{array}{ll}
             \dfrac{(1+a)^p(1+b)^q}{\Gamma(p)\Gamma(q)}t^b\displaystyle\int_t^1 s^{a-b}
             \left(\log\dfrac{1}{s}\right)^{p-1}
                \left(\log\dfrac{s}{t}\right)^{q-1} \dfrac{ds}{s}  , & \hbox{$a\neq b$;} \\
             \dfrac{(1+a)^{p+q}}{\Gamma(p+q)}t^{a}\left(\log\dfrac{1}{t}\right)^{p+q-1}, & \hbox{$a=b$.}
           \end{array}
         \right.
\end{align}
It can be easily seen that $ \displaystyle \mathcal{L}[\eta]f(z)=\phi_{p,q}(a,b;z)\ast f(z) $
where $\phi_{p,q}(a,b;z)$ is the generalized polylogarithm. Note that, if we take
\begin{align*}
    \lambda(t)=\frac{(1+a)^p}{\Gamma(p)}t^a(\log(1/t))^{p-1},\ \
    a>-1,\ p\geq0,
\end{align*}
the integral transform $V_\lambda$ in this case takes the form
\begin{align*}
    V_\lambda(f)(z)=\frac{(1+a)^p}{\Gamma(p)}\int_0^1(\log\left(\frac{1}{t}\right))^{p-1}t^{a-1}f(tz)dt\
    \ a>-1,\ p\geq0
\end{align*}
which is a particular case of the generalized Komatu operator given in \eqref{eqn: generalized Komatu operator}.
The geometric properties of this particular case is studied in the literature to a great extent.
For example, we refer to \cite{ali-abeer-2012-JMAA} for certain inclusion properties related
to this integral operator.  An open problem related to Pick functions is provided by the second
author in \cite{Berg-open}. Hence it would be interesting to answer the following.
\begin{problem}
Let $f\in\mathcal{A}$, then to find the conditions on the parameters
$p,q,a$ and $b$ so that the generalized komatu operator $\mathcal{L}[\eta]f(z)$ belong to
Caratheodory class.
\end{problem}

\subsection*{Acknowledgement}{This work have been partially supported by
“Council of Scientific and Industrial Research, India”
(grant code: 09/143(0827)/2013-EMR-1).}


\begin{thebibliography}{9}
\bibitem {acharya-thesis-1997} {A. P. Acharya}, {Univalence criteria for analytic funtions and
applications to hypergeometric functions},
Ph.D Thesis, University of W\"{u}rzburg, 1997.


\bibitem{ali-abeer-2012-JMAA}
R. M. Ali, A.O. Badghaish, V. Ravichandran and A. Swaminathan,
Starlikeness of integral transforms
and duality, J. Math. Anal. Appl. {\bf 385} (2012), no.~2, 808--822.

\bibitem{alzer-kwong} H. Alzer\ and\ M. K. Kwong,
"A Variant of the Fej\'{e}r-Jackson inequality", Studia
Scientiarum Mathematicarum Hungarica, vol. 56, no. 4, 2019, p. 500+.


\bibitem{askey-steinig-1974-trig-sum-ams}R. Askey\ and\ J. Steinig, Some positive trigonometric sums,
 Trans. Amer. Math. Soc. {\bf 187} (1974), 295--307.

\bibitem{askey-1996-viet-ineq}R. Askey, Vietoris's inequalities and hypergeometric series, in
 {\it Recent progress in inequalities (Ni\v s, 1996)}, 63--76, Math. Appl., 430,
 Kluwer Acad. Publ., Dordrecht.

\bibitem{kiran-ccs-SIGMA}K. K. Behera, A. Sri Ranga\ and\ A. Swaminathan,
Orthogonal polynomials associated with complementary chain sequences, SIGMA
Symmetry Integrability Geom. Methods Appl. {\bf 12} (2016), Paper No. 075,
17 pp.


\bibitem{belov-1995-sine-sum} A. S. Belov, Mat. Sb. {\bf 186} (1995), no. 4, 21--46;
 translation in Sb. Math. {\bf 186} (1995), no.~4, 485--510.

\bibitem{Berg-open} C. Berg, Open problems, Integral Transforms Spec. Funct. {\bf 26} (2015), no.~2, 90--95.

\bibitem{brown-hewitt-1984-pos-trig-sum}G. Brown\ and\ E. Hewitt, A class of positive trigonometric sums,
   Math. Ann. {\bf 268} (1984), no.~1, 91--122.


\bibitem{brown-koumanods-wang-1998-ultraspherical-poly}G. Brown, S. Koumandos\ and\ K.-Y. Wang,
Positivity of basic sums of ultraspherical polynomials, Analysis (Munich) {\bf 18} (1998), no.~4, 313--331.

\bibitem{brown-dai-wang-2007-ext-viet-ramanujan}G. Brown, F. Dai\ and\ K. Wang, Extensions of
    Vietoris's inequalities.  I, Ramanujan J. {\bf 14} (2007), no.~3, 471--507.

\bibitem{dimitrov-merlo-2002-cons-approx}D. K. Dimitrov\ and\ C. A. Merlo,
Nonnegative trigonometric polynomials, Constr. Approx. {\bf 18} (2002), no.~1, 117--143.

\bibitem{signal-processing-book} D. A. Dumitrescu, Positive trigonometric polynomials
    and signal processing applications, Signals and Communication Technology. Springer, Dordrecht, (2007).

\bibitem{fejer-1931-ultasperical-pol} L. Fejer, Ultrasph\"arikus polynomok
\"osszeg\'er\"ol. Mat. Fiz. Lapok {\bf 38}, 161-164 (1931); Also in:
Gesammelte Arbeiten II, 421-423. Birkh\"auser Verlag, Basel und Stuttgart (1970).

\bibitem{Fernandez-2004}J. J. Fern\'andez - Dur\'an,
Circular distributions based on nonnegative trigonometric sums, Biometrics {\bf 60} (2004), no.~2, 499--503.

\bibitem{fournier-ruscheweyh-integral-operator-1994}
R. Fournier\ and\ S. Ruscheweyh, On two extremal problems related
to univalent functions, Rocky Mountain J. Math. {\bf 24} (1994), no.~2, 529--538.

\bibitem{gasper-1969-nonneg-sum-JMAA} G. Gasper, Nonnegative sums of cosine, ultraspherical and Jacobi polynomials,
 J. Math. Anal. Appl. {\bf 26} (1969), 60--68.

\bibitem{komatu-operator-1987-Kodai}Y. Komatu, On a family of integral operators
related to fractional calculus, Kodai Math. J. {\bf 10} (1987), no.~1, 20--38.

\bibitem{koumandos-2007-ext-viet-ramanujan} S. Koumandos, An extension of Vietoris's
inequalities, Ramanujan J. {\bf 14} (2007), no.~1, 1--38.

\bibitem{koumandos-survey} S. Koumandos, Inequalities for trigonometric sums, in {\it Nonlinear analysis},
387--416, Springer Optim. Appl., 68, Springer, New York.

\bibitem{kwong-2015-JAT}M. K. Kwong, An improved Vietoris sine inequality,
    J. Approx. Theory {\bf 189} (2015), 29--42.

\bibitem{lewis-1979-siam} J. L. Lewis, Applications of a convolution theorem to
Jacobi polynomials, SIAM J. Math. Anal. {\bf 10} (1979), no.~6, 1110--1120.

\bibitem{milovanovi-mitrinovi-rassias-1994-book-polynomials}
G. V. Milovanovi\'c, D. S. Mitrinovi\'c\ and\ Th.\ M. Rassias,
 {\it Topics in polynomials: extremal problems, inequalities, zeros}, World Sci. Publishing, River Edge, NJ, 1994.


\bibitem{saiful-swami-polylog} Saiful R. Mondal\ and\ A. Swaminathan, Geometric
properties of generalized Polylogarithm, Integral Transforms Spec. Funct., {\bf 21}(2010) no.~9, 691-701.

\bibitem{saiful-swami-2011-CAMWA} S. R. Mondal\ and\ A. Swaminathan, On the positivity of certain trigonometric
 sums and their applications, Comput. Math. Appl. {\bf 62} (2011), no.~10, 3871--3883.

\bibitem{saiful-2012-stable-results-in-math}
S. R. Mondal\ and\
A. Swaminathan, Stable functions and extension of
Vietoris' theorem, Results Math. {\bf 62} (2012), no.~1-2, 33--51.

\bibitem{Muir} S. Muir, Subordinate solutions of a differential equation,
Comput. Methods Funct. Theory {\bf 7} (2007), no.~1, 1--11.

\bibitem{ruscheweyh-1978-kakeya-thm-SIAM}S. Ruscheweyh, On the Kakeya-Enestr\"om theorem
and Gegenbauer polynomial sums, SIAM J. Math. Anal. {\bf 9} (1978), no.~4, 682--686.


\bibitem{sangal-swaminathan-geom-sigma-bc} P. Sangal\
and\ A. Swaminathan, Geometric Properties of Ces\`aro
Averaging Operators, J. Complex Anal. {\bf 2017}, Art. ID 6584584, 9 pp.


\bibitem{tkachev-arxiv-2003} V. G. Tkachev, Positive trigonometric polynomials, arXiv:math/0301038v1.

\bibitem{vietoris-1958} L.Vietoris, \"Uber das Vorzeichen gewisser trignometrishcher Summen,
Sitzungsber, Oest. Akad. Wiss. 167, 1958,125–-135.


\end{thebibliography}
\end{document}